\documentclass[a4,11pt]{article}
\usepackage{amsmath,amsfonts,amsthm,latexsym,amssymb,mathrsfs}

\def\A{{\mathcal{A}}}

\def\R{{\cal{R}}}

\def\bmatrix{\left[\begin{array}}
\def\endbmatrix{\end{array}\right]}

\def\d{\begin{definition}}
\def\ed{\end{definition}}
\def\t{\begin{theorem}}
\def\et{\end{theorem}}
\def\l{\begin{lemma}}
\def\el{\end{lemma}}
\def\c{\begin{corollary}}
\def\ec{\end{corollary}}

\newtheorem{theorem}{Theorem}[section]
\newtheorem{lemma}{Lemma}[section]
\newtheorem{corollary}{Corollary}[section]
\newtheorem{definition}{Definition}[section]

\newtheorem{prop}{Proposition}[section]

\newtheorem{thm}[prop]{Theorem}
\newtheorem{rema}[prop] {Remark}
\newtheorem{cor}[prop] {Corollary}

\def\sfstp{{\hskip-1em}{\bf.}{\hskip.5em}}
\begin{document}

\pagestyle{myheadings} \markboth{   }
{ \hskip5truecm BOASSO ET AL.}
\title{\bf ON WEIGHTED REVERSE ORDER LAWS FOR THE MOORE-PENROSE INVERSE AND  $K$-$\hbox{\rm \bf INVERSES}$}
  \author{ Enrico Boasso, Dragana S. Cvetkovi\'c-$\hbox{\rm Ili\' c}$ and Robin $\rm Harte$\\ \normalsize}
\date{   }
\maketitle{}
\setlength{\baselineskip}{14pt}
\begin{abstract} The main objective of this article is
to study several generalizations of the reverse order law
for the Moore-Penrose inverse in ring with involution. \par
\vskip.2truecm
\noindent  \it Keywords: \rm Moore-Penrose inverse, $K$-inverse, reverse order law, ring with involution, prime ring.\par
\vskip.2truecm
\noindent \it 2000 Mathematics Subject Classification: 15A09.

\end{abstract}

\section {\sfstp Introduction}
\
\indent Given a complex matrix $a$, the Moore-Penrose inverse of $a$ is the unique complex matrix
$b$ satisfying the following {\it Penrose equations} (Penrose (1955)):\par
$$
(1) \hskip.1truecm a=aba, \hskip1truecm (2)\hskip.1truecm b=bab, \hskip1truecm (3)\hskip.1truecm (ab)^*=ab, \hskip1truecm (4) \hskip.1truecm (ba)^*=ba.
$$
This generalization of  the inverse of a non-singular square matrix
was first introduced by E. H. Moore, but remained unknown mainly because of Moore's special notation
(see Moore (1920)). The equations (1)-(4) were formulated by Penrose, and they
characterize the same object  considered by E. H. Moore.\par
\indent T. N. Greville first characterized when the product of two complex matrices $a$ and $b$
satisfies the so-called {\it reverse order law} for the Moore-Penrose inverse, that is when
$$
(ab)^{\dag}=b^{\dag}a^{\dag},
$$
where $c^{\dag}$ denotes the Moore-Penrose of a complex matrix $c$ (Greville (1966)); note that  the proofs in Greville (1966) remain valid for pairs $a,b$ of Moore-Penrose invertible $C^*$-algebra elements whose product $ab$ also has a Moore-Penrose inverse. In this context, in Boasso (2006) several other conditions equivalent to the reverse order law for the Moore-Penrose inverse
were proved.\par

\indent In the framework of rings with involution, J. J. Koliha, D. S. Djordjevi\'c and D. S. Cvetkovi\'c  extended the characterization in  Greville (1966) under the additional assumption of the $*$-left cancellation property of a particular element of the ring (Koliha, Djordjevi\'c, Cvetkovi\'c (2007)).\par

\indent The first objective of the present article is to study the following generalization of the reverse order law for the Moore-Penrose inverse: given a ring with involution $\mathcal{R}$, elements in $\mathcal{R}$ for which $a$, $b$ and $ab$ are Moore-Penrose invertible, and an element $c\in \mathcal{R}$ which commutes with $b$ and $b^*$, characterize when the following identity holds:
$$
(ab)^{\dag}=c b^{\dag}a^{\dag}.
$$
This identity and others presented in section 2 will be called \it weighted reverse order laws \rm for the Moore-Penrose inverse.\par

Naturally,
when $c=e$ a characterization of the usual reverse order law is obtained. Furthermore,
other similar generalizations of  the reverse order law for the Moore-Penrose
inverse in rings with involution and in complex algebras with involution will be also considered,
see next section. Note that no additional assumption such as the $*$-cancellation
property for elements of the ring is needed.\par
 
\indent On the other hand, given a $C^*$-algebra $\mathcal{A }$, an element $a\in \mathcal{A }$ and a subset $K\subseteq \{1,2,3,4\}$,
an element $x\in  \mathcal{A }$ is said to be a \it $K$-inverse of $a$\rm, if $x$ satisfies the Penrose equation $(j)$ for each $j\in K$. Several
reverse order laws for $K$-inverses of products of two $C^*$-algebra elements were characterized by D. S. Cvetkovi\' c-Ili\' c and R. E. Harte.
The second objective of this work is to extend some of the results in Cvetkovi\' c-Ili\' c, Harte (2011)
to weighted reverse order laws in rings with involution, see section 3.\par
\vskip.2truecm
\indent  Before going on,  several definitions and some notation will be recalled.\par

\indent Let $\mathcal{R}$ be an associative ring with unit element $e$.
The ring $\mathcal{R}$ is said to be a \it prime ring\rm, if whenever elements $a$ and $b\in\R$ satisfy $a\mathcal{R }b=\{0\}$, then $0\in\{a,b\}$
(see McCoy (1949)). For example, given $n\in \mathbb N$, the ring of square matrices $\mathbb C^{n\times n}$ is prime, see Lemma 3 in Baksalary, Baksalary (2005). 
It is not difficult to prove that the same is true when $\mathcal{A}\subseteq L(X)$ is a subalgebra of the Banach algebra of all bounded operators defined on the Banach space $X$ which contains the ideal
of finite rank operators. 
In the case of general Banach algebras, prime, ultraprime and spectrally prime algebras were considered in 
Harte, Hern\'andez (1998).\par

\indent An element $a\in\mathcal{R}$ is said to be \it group invertible \rm if there exists $b\in\mathcal{R}$
such that
$$
aba=a,\hskip1truecm bab=b,\hskip1truecm ab=ba.
$$
It is well known that if $a\in \mathcal{R}$ is group invertible, then there is only one group inverse of $a$ (Mosi\' c, Djordjevi\' c (2009)), which will be denoted by $a^{\sharp}$.\par
\indent An involution $^*\colon \mathcal{R}\to \mathcal{R}$ is an anti-isomorphism of degree $2$, that is
$$
(a^*)^*=a,\hskip1truecm (a+b)^*=a^*+b^*,\hskip1truecm (ab)^*=b^*a^*.
$$

\indent Given $\mathcal{R}$ a ring with involution, an element $a\in \mathcal{R}$ is said to be \it Hermitian \rm
if $a=a^*$, and $a$ is said to be \it Moore-Penrose invertible \rm if there exists $b\in \mathcal{R}$ such that $a$ and $b$ satisfy the Penrose
equations presented above.

\indent It is well known that given $a\in  \mathcal{R}$, there is at most one Moore-Penrose inverse of $a$, see  Roch, Silbermann (1999). When the Moore-Penrose inverse of $a\in  \mathcal{R}$ exists,  it will be denoted, as before, by $a^{\dag}$. In addition, $\mathcal{R}^{\dag}$ will stand for the set of all Moore-Penrose invertible elements of $a\in\mathcal{R}$.
Note that if $a\in \mathcal{R}^{\dag}$, then $aa^{\dag}$  and $a^{\dag}a$ are hermitian idempotents.
What is more,  if $a\in \mathcal{R}^{\dag}$, then $a^{\dag}\in \mathcal{R}^{\dag}$ and $(a^{\dag})^{\dag}=a$.
Moreover, it is easy to prove that  $a\in \mathcal{R}^{\dag}$ if and only if  $a^*\in \mathcal{R}^{\dag}$. Furthermore,
in this case, $(a^*)^{\dag}=(a^{\dag})^*$. In what follows $(a^{\dag})^*$  will be denoted by $a^{\dag*}$.
 \par

\indent Given $a\in \mathcal{R}$ and $K\subseteq \{1,2,3,4\}$, $x\in\mathcal{R}$  will be said to be a
\it $K$-inverse of $a$\rm, if $x$ satisfies the same condition recalled above for $C^*$-algebra elements.
The set of all $K$-inverses of a given $a\in \mathcal{R}$ will be denoted by $aK$.\par

\indent Finally, if $p$ and $q$ are idempotents in $\mathcal{R}$, then an arbitrary $x\in \mathcal{R}$
can be represented as a $2\times 2$ matrix over $\mathcal{R}$; specifically
 $$
x=\bmatrix{cc} x_1&x_2\\x_3&x_4\endbmatrix_{p,q},
$$
where $x_1=pxq$, $x_2=px(e-q)$, $x_3= (e-p)xq$ and $x_4=(e-p)x(e-q)$.
Note that $x=x_1+x_2+x_3+x_4$.\par

\section {\sfstp Weighted reverse order laws for the Moore-Penrose inverse}
\
\indent We begin by presenting an equivalent formulation for the Moore-Penrose inverse. Although its proof is not difficult (Penrose (1955)),
it will be used below, and hence we reproduce it here. \par
\markright{\hskip4truecm WEIGHTED REVERSE ORDER LAWS }
\begin{prop}\label{prop1}Let  $\mathcal{R}$ be a ring with involution
and consider $a\in   \mathcal{R}$. Then, the following statements are equivalent:\par
\noindent (i) $b\in   \mathcal{R}$ is the Moore-Penrose inverse of $a$;\par
\noindent (ii) $a= aa^*b^*$ and $b^*=abb^*$.
\end{prop}
\begin{proof} If $b= a^{\dag}$ then, since $(ab)^*=ab$,
$$
a=aba= a(ba)=a(ba)^* =aa^*b^*,\hskip1truecm
b^*=(bab)^*= (ab)^*b^*=abb^*.
$$

\indent On the other hand, if statement (ii) holds,
then
$$
ba=baa^*b^*=ba(ba)^*, \hskip1truecm
(ab)^*= b^*a^*=abb^*a^*= ab(ab)^*,
$$
equivalently, $ab$ and $ba$ are hermitian idempotents.
However, according to statement (ii),
$$
a=a(ba)^*=aba,\hskip2truecm b=b(ab)^*=bab.
$$
\end{proof}

\indent The following proposition will extend to rings with involution a well
known result concerning the Moore-Penrose inverse of $\mathbb C^*$-algebra elements,
see Theorem 7 in Harte, Mbekhta (1992).\par

\begin{prop}\label{prop2}Let  $\mathcal{R}$ be a ring with involution
and consider $a\in \mathcal{R}^{\dag}$ and  $c\in   \mathcal{R}$. Necessary and sufficient condition for $c$ to commute with $a$
and $a^*$ is that $c$ commutes with $a^{\dag}$ and  $a^{\dag*}$.
\end{prop}
\begin{proof} Let $a\in \mathcal{R}^{\dag}$. Then, according to Theorem 5.3
in  Koliha, Patr\'{\i}cio (2002),
$(a^*a)^{\sharp}$ exists. Moreover,
$$
a^{\dag}=(a^*a)^{\sharp}a^*.
$$
\indent If $c$ commutes with $a$ and $a^*$, then $c$
commutes with $a^*a$. In addition, since $a^*a$ is group invertible,
$c$ commutes with $(a^*a)^{\sharp}$, see Mosi\' c, Djordjevi\' c (2009). Therefore $c$ commutes with  $a^{\dag}
=(a^*a)^{\sharp}a^*$.\par

\indent In addition, since $a^*\in  \mathcal{R}^{\dag}$
and $(a^*)^*=a$, according to what has been proved, $c$ commutes
with $(a^*)^{\dag}=a^{\dag*}$.\par

\indent On the other hand, if $c$ commutes with $a^{\dag}$ and  $a^{\dag*}$, then since  $(a^{\dag})^{\dag}=a$ and  $(a^{\dag})^{\dag*}=a^*$,  $c$ commutes with $a$ and $a^*$.
\end{proof}

\indent Let   $\mathcal{R}$ be a ring with involution and
consider  $a, b\in\mathcal{R}^{\dag}$. Define
$$p=bb^{\dag},\hskip.5truecm q=a^{\dag}a^{\dag*},\hskip.5truecm
r=bb^*,\hskip.5truecm s=a^{\dag}a.
$$
 Clearly $p$, $q$, $r$ and $s$ are hermitian elements.
Moreover, according to Proposition \ref{prop1},
$$a=as,\hskip.3truecm  a^{\dag*}=aq,\hskip.3truecm
b=rb^{\dag*},\hskip.3truecm b^{\dag*}=pb^{\dag*}.
$$
\noindent Note that $p$, $q$, $r$ and $s$ are blanket notations for this section.\par
\indent In the following theorems several weighted reverse order laws for
the Moore-Penrose inverse will be presented. Note that when
$c=e$, then a characterization of the usual reverse order law in rings
with involution is obtained.\par
\markright{\hskip5truecm BOASSO ET AL. }
\begin{thm}\label{thm3} Let $\mathcal{R}$ be a ring with involution.
Consider $a,b \in\mathcal{R}^{\dag}$ such that $ab\in \mathcal{R}^{\dag}$,
and $c\in \mathcal{R}$ such that $c$ commutes with $b$ and $b^*$.
Then, the following statements are equivalent:\par
\noindent (i) $(ab)^{\dag}=c b^{\dag}a^{\dag}$;\par
\noindent (ii) $a(cpq-qp)b^{\dag*}c^*=0$ and $a(rsc^*-sr)b^{\dag*}=0$;\par
\noindent (iii) $scpqpc^*=qpc^*$ and  $srspc^*=sr$.\par
\end{thm}
 \begin{proof} In first place, note that according to Proposition \ref{prop1},
\begin{align*}
&a=aa^*a^{\dag*},& &b=bb^*b^{\dag*},& &ab= ab(ab)^*(ab)^{\dag*},&\\
&a^{\dag*}=aa^{\dag}a^{\dag*},&
&b^{\dag*}=bb^{\dag}b^{\dag*},&
&(ab)^{\dag*}=ab(ab)^{\dag}(ab)^{\dag*}.&\\
\end{align*}

\indent (i) $\Rightarrow$ (ii). Suppose that $(ab)^{\dag}=c b^{\dag}a^{\dag}$.
Then, since $(ab)^*$ = $b^*a^*$ and
$(ab)^{\dag*}$ = $(cb^{\dag}a^{\dag})^*$ = $a^{\dag*}b^{\dag*}c^*$,
$$
ab= abb^*a^*a^{\dag*}b^{\dag*}c^*,\hskip1cm
a^{\dag*}b^{\dag*}c^*= abcb^{\dag}a^{\dag}a^{\dag*}b^{\dag*}c^*,
$$
which, since $c$ and $b$ commute, can be written as
$$
asrb^{\dag*}=arsb^{\dag*}c^*,\hskip1.5cm aqpb^{\dag*}c^*=
acpqb^{\dag*}c^*.
$$
\indent However, according to Proposition \ref{prop2} these identities are equivalent to
$$
a(cpq-qp)b^{\dag*}c^*=0, \hskip2cm a(rsc^*-sr)b^{\dag*}=0.
$$

\indent (ii) $\Rightarrow$ (iii). If the second statement holds,
then
$$
a^{\dag}acpqb^{\dag*}c^*b^*= a^{\dag}aqpb^{\dag*}c^*b^*,\hskip1cm
a^{\dag}arsc^*b^{\dag*}b^*= a^{\dag}asrb^{\dag*}b^*.
$$
\indent However, since $c$ commutes with $b$ and $b^{\dag}$,
$$
a^{\dag}acpqb^{\dag*}b^*c^*= a^{\dag}aqpb^{\dag*}b^*c^*,\hskip1cm
a^{\dag}arsb^{\dag*}b^*c^*= a^{\dag}asrb^{\dag*}b^*.
$$

\indent What is more, according again to Proposition \ref{prop1} and to the fact that $s=s^*$ and $p=p^*$, these equations can be rewritten as
$$
scpqpc^*=a^{\dag}(aa^{\dag}a^{\dag*})b(b^{\dag}b^{\dag*}b^*)c^*=a^{\dag}a^{\dag*}bb^{\dag}c^*=qpc^*,
$$
$$
srspc^*=(a^{\dag}aa^*)a^{\dag*}(bb^*b^{\dag*})b^*=a^*a^{\dag*}bb^*=sr.
$$
\indent (iii) $\Rightarrow$ (i). Suppose that statement (iii) holds. Then, since $p=p^*$, $s=s^*$ and $b$ and $c$ commute,
$$
 a^{\dag}abcb^{\dag}a^{\dag}a^{\dag*}b^{\dag*}b^*c^*=a^{\dag}a^{\dag*}bb^{\dag}c^*,
$$
$$
 a^{\dag}abb^*a^*a^{\dag*}b^{\dag*}b^*c^*=a^*a^{\dag*}bb^*.
$$
\indent Moreover, since  $c$ and $b^{\dag}$ commute,
$$
(aa^{\dag}a)bcb^{\dag}a^{\dag}a^{\dag*}(b^{\dag*}b^*b^{\dag*})c^*=(aa^{\dag}a^{\dag*})(bb^{\dag}b^{\dag*})c^*,
$$
$$
(aa^{\dag}a)bb^*a^*a^{\dag*}(b^{\dag*}b^*b^{\dag*})c^*=(aa^*a^{\dag*})(bb^*b^{\dag*}).
$$
\indent However, according to Proposition \ref{prop1}, these equations are
equivalent to
$$
 ab(cb^{\dag}a^{\dag})(cb^{\dag}a^{\dag})^* = (cb^{\dag}a^{\dag})^*,
$$
$$
 ab(ab)^*(cb^{\dag}a^{\dag})^*=ab.
$$
\indent Therefore,
$$
(ab)^{\dag}=c b^{\dag}a^{\dag}.
$$
\end{proof}

\indent As an application of Theorem \ref{thm3}, other generalizations of the reverse order law can be characterized.\par
\markright{\hskip4truecm WEIGHTED REVERSE ORDER LAWS }
\begin{thm}\label{thm4} Let $\mathcal{R}$ be a ring with involution.
Consider $a,b \in\mathcal{R}^{\dag}$ such that $ab\in \mathcal{R}^{\dag}$,
and $c\in \mathcal{R}$ such that $c$ commutes with $a$ and $a^*$.
Then, the following statements are equivalent:\par
\noindent (i) $(ab)^{\dag}= b^{\dag}a^{\dag}c$;\par
\noindent (ii) $b^*(c^*sr^{\dag}-r^{\dag}s)a^{\dag}c=0$ and $b^*(q^{\dag}pc-pq^{\dag})a^{\dag}=0$;\par
\noindent (iii) $pc^*sr^{\dag}sc=r^{\dag}sc$ and  $pq^{\dag}psc=pq^{\dag}$.\par
\end{thm}

\begin{proof}Recall that given $h\in\mathcal{R}$, necessary and sufficient for $h$ to belong to
 $\mathcal{R}^{\dag}$
is that $h^*\in  \mathcal{R}^{\dag}$, (see Theorem 5.4 in  Koliha, Patr\'{\i}cio (2002)). Moreover, in this case $(h^*)^{\dag}=(h^{\dag})^*$.
It is not difficult to prove that the identity $(ab)^{\dag}= b^{\dag}a^{\dag}c$
is equivalent to
$$
(b^*a^*)^{\dag}= c^*(a^*)^{\dag}(b^*)^{\dag}.
$$

\indent On the other hand, denote by $p_1$, $q_1$, $r_1$ and $s_1$
the elements of $\mathcal{R}$ corresponding to $p$, $q$, $r$ and $s$
defined using $b^*a^*$ instead of $ab$. Then, it is easy to prove that
$$
p_1= s,\hskip1truecm s_1= p.
$$
\indent In addition, according to  the proof of Theorem 5.3  in Koliha, Patr\'{\i}cio (2002),
$$
q_1=r^{\dag}, \hskip1truecm r_1=q^{\dag}.
$$
\par
\indent To conclude the proof, apply Theorem 2.3 to $b^*$,  $a^*$, $b^*a^*$ and $c^*$
in place of $a$, $b$, $ab$ and $c$.
\end{proof}

\begin{thm}\label{thm5} Let $\mathcal{R}$ be a ring with involution.
Consider  $a,b \in\mathcal{R}^{\dag}$ and $c\in \mathcal{R}$
such that $cab\in \mathcal{R}^{\dag}$.
Then, if $c$ commutes with  $a$ and  $a^*$, the following statements are equivalent:\par
\noindent (i) $(cab)^{\dag}= b^{\dag}a^{\dag}$;\par
\noindent (ii) $b^{\dag}(csr-rs)a^*c^*=0$ and $b^{\dag}(qpc^*-pq)a^*=0$;\par
\noindent (iii) $pcsrsc^*=rsc^*$ and  $pqpsc^*=pq$.\par

\end{thm}

\begin{proof} Note that $cab\in \mathcal{R}^{\dag}$ and $(cab)^{\dag}= b^{\dag}a^{\dag}$
if and only if $b^{\dag}a^{\dag}\in\mathcal{R}^{\dag}$ and $( b^{\dag}a^{\dag})^{\dag}= cab$.\par

\indent As in the proof of Theorem \ref{thm4}, $p_2$, $q_2$, $r_2$ and $s_2$
 denote the elements of $\mathcal{R}$ corresponding to $p$, $q$, $r$ and $s$
defined using $b^{\dag}a^{\dag}$ instead of  $ab$. Then, it is easy to prove that
$$
p_2=s,\hskip.5truecm q_2=r,\hskip.5truecm r_2=q,\hskip.5truecm s_2=p.
$$
\indent To conclude the proof, apply Theorem 2.3 to $b^{\dag}$,  $a^{\dag}$, $b^{\dag}a^{\dag}$
and $c$ in place of $a$, $b$, $ab$ and $c$.
\end{proof}
\markright{\hskip5truecm BOASSO ET AL. }
\begin{thm}\label{thm6} Let $\mathcal{R}$ be a ring with involution.
Consider $a,b \in\mathcal{R}^{\dag}$ and $c\in \mathcal{R}$
such that   $abc\in \mathcal{R}^{\dag}$.
Then, if $c$ commutes with $b$ and $b^*$, the following statements are equivalent:\par
\noindent (i) $(abc)^{\dag}= b^{\dag}a^{\dag}$;\par
\noindent (ii)  $a^{\dag*}(c^*pq^{\dag}-q^{\dag}p)bc=0$ and  $a^{\dag*}(r^{\dag}sc-sr^{\dag})b=0$;\par
\noindent (iii) $sc^*pq^{\dag}pc=q^{\dag}pc$ and  $sr^{\dag}spc=sr^{\dag}$.\par
\end{thm}
\begin{proof} It is easy to prove that the first statement is equivalent to
$$
(a^{\dag*}b^{\dag*})^{\dag}= c^*b^*a^*.
$$

\indent As in Theorem  \ref{thm4} and Theorem \ref{thm5}, denote by
$p_3$, $q_3$, $r_3$ and $s_3$
the elements of $\mathcal{R}$ corresponding to $p$, $q$, $r$ and $s$
defined using  $a^{\dag*}b^{\dag*}$ instead of  $ab$. Then, using the proof of Theorem 5.3 in   Koliha, Patr\'{\i}cio (2002), we prove that
$$
p_3=p,\hskip.5truecm q_3=q^{\dag},\hskip.5truecm r_3=r^{\dag},\hskip.5truecm s_3=s.
$$
\indent To conclude the proof, apply Theorem \ref{thm3} to  $a^{\dag*}$, $b^{\dag*}$,
 $a^{\dag*}b^{\dag*}$ and $c^*$  in place of $a$, $b$, $ab$ and $c$.
\end{proof}

\indent Specializing to the case  of an algebra with involution over the complex numbers $\mathbb C$ ($\mathcal{R}=\mathcal{A}$),
 $\overline{\lambda}\in \mathbb C$ will stand for  the \it complex conjugate \rm of $\lambda\in \mathbb C$. Note that $(\lambda a)^*=\overline{\lambda}a^*$ for any $a$
in the algebra. In particular we can now allow the element  $c\in \mathcal{A}$ be a scalar multiple of the identity, i.e.,
$c=\lambda e$.\par

 \begin{cor}\label{cor7} Let $\mathcal{A}$ be an algebra with involution over $\mathbb C$.
Consider $a,b \in\mathcal{A}^{\dag}$ such that  $ab\in \mathcal{A}^{\dag}$,
and $\lambda\in \mathbb C$.
Then, the following statements are equivalent:\par
\noindent (i) $(ab)^{\dag}=\lambda b^{\dag}a^{\dag}$;\par
\noindent (ii)  $a(\lambda pq-qp)b^{\dag*}=0$ and  $a(rs\overline{\lambda}-sr)b^{\dag*}=0$;\par
\noindent (iii) $\lambda spqp=qp$ and  $\overline{\lambda}srsp=sr$.\par
\end{cor}
\begin{proof} Apply Theorem \ref{thm3}.
\end{proof}
\begin{rema}\label{rem7} \rm  Let $\mathcal{A}$ be a $C^*$-algebra and consider $a,b \in\mathcal{A}^{\dag}$
such that  $ab\in \mathcal{A}^{\dag}$. Let $p$, $q$, $r$ and $s$ be the elements of $\mathcal{A}$
defined before Theorem \ref{thm3}. Recall that, according to Remark 3.5 in Boasso (2006)
or  Greville (1966), $(ab)^{\dag}=b^{\dag}a^{\dag}$ if and only if $rs=sr$ and $pq=qp$. Note that according to
Theorem 7 in  Harte, Mbekhta (1992), $p$ and $q$ commute (respectively $r$ and $s$ commute)
if and only if $p$ and $q^{\dag}$ commute (respectively $s$  and $r^{\dag}$ commute).
What is more, these statements are equivalent to the at first sight weaker
conditions of Theorems 3.1-3.4 in Boasso (2006). \par

\indent When $\mathcal{R}$ a ring with involution, $a,b \in \mathcal{R}^{\dag}$ and
$(e-a^{\dag}a)b$ is left $*$-cancellable, necessary and
sufficient for  $ab$ to belong to $\mathcal{R}^{\dag}$ and $(ab)^{\dag}=b^{\dag}a^{\dag}$
is that $rs=sr$ and $pq^{\dag}=q^{\dag}p$ (see Theorem 3 in Koliha, Djordjevi\' c, Cvetkovi\' c (2007)). Note also that according to
 Proposition \ref{prop2},  $p$ and $q$ commute (respectively $r$ and $s$ commute)
if and only if $p$ and $q^{\dag}$ commute (respectively $s$  and $r^{\dag}$ commute).
Therefore, considering $c=e$, the conditions presented in Theorem \ref{thm3} are weaker
than the ones in Theorem 3 in Koliha, Djordjevi\' c, Cvetkovi\' c (2007)
and to prove them the cancellation property is not necessary. In particular,
while all the aforementioned results are equivalent in $C^*$-algebras,
in the case of  rings with involution, according to the characterization of Theorem \ref{thm3},
if the reverse order law is satisfied by  $a$ and $b$, the identities $rs=sr$ and $pq^{\dag}=q^{\dag}p$
need not to be satisfied. According to Theorem 3 in  Koliha, Djordjevi\' c, Cvetkovi\' c (2007), these equalities are satisfied
when the cancellation property is assumed.
\end{rema}

\section{\sfstp Weighted reverse order laws for $K$-inverses in prime rings}
\
\markright{\hskip4truecm WEIGHTED REVERSE ORDER LAWS }
\indent In this section, $ \mathcal{R}$ will be a prime ring with involution and $K\subseteq \{1,2,3,4\}$. For  $a,b\in \mathcal{R}$,
several weighted reverse order laws for  $K$-inverses of  $ab$ will be characterized. First we will present  some preliminary facts.
\par

\begin{rema}\label{rema8}\rm

Consider $a, b\in\mathcal{R}^{\dag}$ and $c\in \mathcal{R}$ such that $c$ commutes with $a$ and $a^*$.
Let  $p=bb^\dag$, $q=b^\dag b$ and $r=aa^\dag$. We have that $b=\bmatrix{cc} b&0\\0&0\endbmatrix_{p,q}$ and $a=\bmatrix{cc}a_1&a_2\\0&0\endbmatrix_{r,p}$. An arbitrary $b^{(1,3)}\in b\{1,3\}$ has the form
$b^{(1,3)}=\bmatrix{cc}b^\dag&0\\u&v\endbmatrix_{q,p}$, for some $u\in (e-q)\A p$ and $v\in (e-q)\A (e-p)$, and an arbitrary $a^{(1,3)}$ has the form
$a^{(1,3)}=a^\dag +(e-a^\dag a)x$, for some $x=\bmatrix{cc}x_1&x_2\\x_3&x_4\endbmatrix_{p,r}\in  \mathcal{R}$, i.e., $a^{(1,3)}=\bmatrix{cc}z_1&z_2\\z_3&z_4\endbmatrix_{p,r}$, where

\parbox{10.5cm}{\begin{eqnarray*}
&&z_1=a_1^*d^\dag +(e-a_1^*d^\dag a_1)x_1-a_1^*d^\dag a_2x_3,\\
&&z_2=(e-a_1^*d^\dag a_1)x_2-a_1^*d^\dag a_2x_4,\\
&&z_3=a_2^*d^\dag -a_2^*d^\dag a_1x_1+(e-a_2^*d^\dag a_2)x_3,\\
&&z_4=-a_2^*d^\dag a_1x_2+(e-a_2^*d^\dag a_2)x_4.
\end{eqnarray*}}     \hfill
\parbox{1cm}{\vskip -1.1cm$$\eqno{(3.1)}$$}

Also, $a^\dag=a^*(aa^*)^\dag=\bmatrix{cc} a_1^*d^\dag &0\\a_2^*d^\dag&0\endbmatrix_{p,r}$ (Theorem 5.3 in  Koliha, Patr\'{\i}cio (2002))
 where $d=aa^*=a_1a_1^*+a_2a_2^*$
and $d^\dag =(aa^*)^\dag$.

If  $c$ commutes with $a$ and $a^*$, it follows that $c=\bmatrix{cc}c_1&0\\0&c_2\endbmatrix_{r,r}$.

\end{rema}

\markright{\hskip5truecm BOASSO ET AL. }
\begin{thm}\label{thm9}Let $\mathcal{R}$ be a prime ring with involution.
Consider  $a, b\in\mathcal{R}^{\dag}$ and $c\in \mathcal{R}$ such that $c$ commutes with  $a$ and  $a^*$.
Then, following statements are equivalent:

\item $(i)$ $b\{1,3\}\cdot a\{1,3\}\cdot c\subseteq (ab)\{1,3\}$;

\item $(ii)$ $b^\dag a^\dag c\in ab\{1,3\}$, $b^\dag a^\dag \in ab\{1\}$ and  $a^\dag\in a(e-bb^\dag)\{1\}$.
\end{thm}

\begin{proof} Note that the case $ab=0$ is trivial. Hence, we will consider the case $ab\neq 0$.\par
Clearly $b\{1,3\}\cdot a\{1,3\}\cdot c\subseteq (ab)\{1,3\}$ is equivalent to the fact that for any  $a^{(1,3)}\in a\{1,3\}$ and any $b^{(1,3)}\in b\{1,3\}$:
$$
b^{(1,3)}a^{(1,3)}c\in(ab)\{1,3\}.
\eqno{(3.2)}
$$

 Using the matrix representations considered in Remark \ref{rema8}, we have that  $(3.2)$ is equivalent to

$$
\hbox{\rm (i)}\hskip.1truecm a_1z_1c_1a_1=a_1,\hskip1truecm  \hbox{\rm (ii)}\hskip.1truecma_1z_2c_2=0,
\hskip1truecm  \hbox{\rm (iii)}\hskip.1truecm (a_1z_1c_1)^*=a_1z_1c_1,
\eqno{(3.3)}
$$
where $z_1$ and $z_2$ are defined by $(3.1)$. Now, using $(3.1)$, the identity  $(3.3)$(i) is equivalent to
\begin{eqnarray*}
a_1a_1^*d^\dag c_1a_1+(a_1-a_1a_1^*d^\dag a_1)x_1c_1a_1-a_1a_1^*d^\dag a_2x_3c_1a_1=a_1.
\end{eqnarray*}
Since, $x_1$ and $x_3$ are arbitrary elements from appropriate subalgebras,  $(3.3)$(i)
is equivalent to
$$
\hbox{\rm (i)}\hskip0.2truecm a_1a_1^*d^\dag c_1a_1=a_1,\hskip0.5truecm \hbox{\rm (ii)}\hskip0.2truecm (a_1-a_1a_1^*d^\dag a_1)x_1c_1a_1=0,\hskip0.5truecm
\hbox{\rm (iii)}\hskip0.2truecm a_1a_1^*d^\dag a_2x_3c_1a_1=0.
\eqno{(3.4)}
$$

What is more, since  $a_1=rap$, $x_1=pxr$ and $c_1=rcr$, $(3.4)$(ii) is equivalent to
$$
 (a_1-a_1a_1^*d^\dag a_1)xc_1a_1=0,
$$
where $x\in\mathcal{R}$ is arbitrary. However, since $\mathcal{R}$ is a prime ring,
  $a_1-a_1a_1^*d^\dag a_1=0$ or $c_1a_1=0$.
\par

\indent Similarly, from   $(3.4)$(iii), we get that
 $a_1a_1^*d^\dag a_2=0$ or $c_1a_1=0$.\par
\indent Note that the case $c_1a_1=0$ implies that $ab=0$, which is not possible.

 Therefore, $c_1a_1\neq0$, and equation $(3.3)$(i) is equivalent to
$$
\hbox{\rm (i)}\hskip.1truecm a_1a_1^*d^\dag c_1a_1=a_1,\hskip1truecm \hbox{\rm (ii)}\hskip.1truecm a_1-a_1a_1^*d^\dag a_1=0,\hskip1truecm
\hbox{\rm (iii)}\hskip.1truecm a_1a_1^*d^\dag a_2=0,
\eqno{(3.5)}
$$

i.e.,
$$
\hbox{\rm (i)}\hskip.1truecm b^\dag a^\dag c\in ab\{1\} ,\hskip1truecm \hbox{\rm (ii)}\hskip.1truecm b^\dag a^\dag \in ab\{1\},\hskip1truecm
\hbox{\rm (iii)}\hskip.1truecm a^\dag\in a(e-bb^\dag)\{1\}.
\eqno{(3.6)}
$$

\indent Now, $(3.5)$ imply that  $a_1z_2=0$ and the fact that  $a_1z_1c_1=a_1a_1^*d^\dag c_1$ is hermitian is  equivalent to the fact that $abb^\dag a^\dag c$ is hermitian, i.e. $ b^\dag a^\dag c\in ab\{3\}$.
\end {proof}

From the proof of Theorem \ref{thm9} it follows that under the assumption  $a_2=a(e-bb^\dagger)\in \mathcal{R}^{\dag}$, condition $(3.5)(ii)$ implies condition $(3.5)(iii)$:
$$
a_1-a_1a_1^*d^\dag a_1=0 \Rightarrow a_1-dd^\dag a_1+a_2a_2^*d^\dag a_1=0\Rightarrow a_2^*d^\dag a_1=0,
$$
so we get the following corollary.

\begin{cor}\label{cor10}Let $\mathcal{R}$ be a prime ring with involution.
Consider  $a, b\in\mathcal{R}^{\dag}$ such that $a(e-bb^\dagger)\in \mathcal{R}^{\dag}$  and let $c\in \mathcal{R}$ such that $c$ commutes with  $a$ and  $a^*$.
Then, following statements are equivalent:

\item $(i)$ $b\{1,3\}\cdot a\{1,3\}\cdot c\subseteq (ab)\{1,3\}$;

\item $(ii)$ $b^\dag a^\dag c\in ab\{1,3\}$, $b^\dag a^\dag \in ab\{1\}$.
\end{cor}

\indent In the following theorem, for given $M\subseteq \mathcal{R}$, $M^*$ will stand for the set
of all adjoint elements of $M$, i.e., $M^*=\{x^*\colon x\in M\}$.\par

\begin{thm}\label{thm11}Let $\mathcal{R}$ be a prime ring with involution.
Consider  $a, b\in\mathcal{R}^{\dag}$ and $c\in \mathcal{R}$ such that $c$ commutes with $b$ and $b^*$.
Then, following statements are equivalent:

\item $(i)$ $c\cdot b\{1,4\}\cdot a\{1,4\}\subseteq (ab)\{1,4\}$;

\item $(ii)$ $cb^\dag a^\dag \in ab\{1,4\}$, $b^\dag a^\dag \in ab\{1\}$ and $b^\dag\in (e-a^\dag a)b\{1\}$.
\end{thm}

\begin{proof} Note that for  given $x\in  \mathcal{R}$, $(x\{1,4\})^*=x^*\{1,3\}$. Therefore, the first statement
is equivalent to $  a^*\{1,3\}\cdot b^*\{1,3\}\cdot c^*\subseteq (b^*a^*)\{1,3\}$. Now  apply Theorem \ref{thm9}.
\end{proof}

\indent As in the case of Theorem \ref{thm9}, the following corollary can be deduced from Theorem \ref{thm11}.\par

\begin{cor} \label{cor12}Let $\mathcal{R}$ be a prime ring with involution.
Consider  $a, b\in\mathcal{R}^{\dag}$ such that $(e-a^\dagger a)b\in \mathcal{R}^{\dag}$  and let  $c\in \mathcal{R}$ such that $c$ commutes with  $b$ and  $b^*$.
Then, following statements are equivalent:

\item $(i)$ $c\cdot b\{1,4\}\cdot a\{1,4\}\subseteq (ab)\{1,4\}$;

\item $(ii)$ $cb^\dag a^\dag \in ab\{1,4\}$, $b^\dag a^\dag \in ab\{1\}$ .
\end{cor}

\markright{\hskip4truecm WEIGHTED REVERSE ORDER LAWS }
\begin{thm}\label{thm13}Let $\mathcal{R}$ be a prime ring with involution.
Consider  $a, b\in\mathcal{R}^{\dag}$ and $c\in \mathcal{R}$ such that $c$ commutes with  $a$ and  $a^*$.
Then, following statements are equivalent:

\item $(i)$ $b\{1,3\}\cdot a\{1,3\}\subseteq (cab)\{1,3\}$;

\item $(ii)$ $b^\dag a^\dag \in (cab)\{1,3\}$, $cab=cabb^\dag a^\dag ab$ and $ca(e-bb^\dag) a^\dag a(e-bb^\dag)=ca(e-bb^\dag)$.
\end{thm}

\begin{proof} Using arguments similar to the ones in the proof of Theorem \ref{thm9}, it is not difficult to prove
that the first statement of the theorem is equivalent to the following equations.
\begin{align*}
&\hbox{\rm (i)}\hskip.05truecm c_1a_1a_1^*d^\dag c_1a_1=c_1a_1,& &\hbox{\rm (ii)}\hskip.05truecm c_1a_1a_1^*d^\dag =(c_1a_1a_1^*d^\dag)^*,&\\
&\hbox{\rm (iii)}\hskip.05truecm c_1a_1-c_1a_1a_1^*d^\dag a_1=0,& & \hbox{\rm (iv)}\hskip.05truecm c_1a_1a_1^*d^\dag a_2=0.&\\
\end{align*}
\indent The first two equations are equivalent to $b^\dag a^\dag \in (cab)\{1,3\}$,
the third to  $cabb^\dag a^\dag ab= cab$ and the fourth to $ca(e-bb^\dag) a^\dag a(e-bb^\dag)=ca(e-bb^\dag)$.
\end{proof}

\markright{\hskip5truecm BOASSO ET AL. }
\begin{thm}\label{thm14}Let $\mathcal{R}$ be a prime ring with involution.
Consider  $a, b\in\mathcal{R}^{\dag}$ and $c\in \mathcal{R}$ such that $c$ commutes with $b$ and $b^*$.
Then, following statements are equivalent:

\item $(i)$ $b\{1,4\}\cdot a\{1,4\}\subseteq (abc)\{1,4\}$;

\item $(ii)$ $b^\dag a^\dag \in (abc)\{1,4\}$,  $ab= abb^\dag a^\dag abc$ and $ (e-a^\dag a)bc=  (e-a^\dag a)bb^\dag  (e-a^\dag a)bc$.
\end{thm}
\begin{proof} As in Theorem \ref{thm11}, since given $x\in  \mathcal{R}$,
$(x\{1,4\})^*=x^*\{1,3\}$, the first statement is equivalent to $a^*\{1,3\}\cdot b^*\{1,3\}\subseteq (c^*b^*a^*)\{1,3\}$. Now apply
Theorem \ref{thm13}.
\end{proof}

\indent Next some characterizations of  reverse order laws for $K$-inverses in $C^*$-algebras will be extended
to the context of the present work.\par

\begin{thm}\label{thm15} Let $\mathcal{R}$ be a ring with involution.
Consider $a,b \in\mathcal{R}^{\dag}$ such that  $ab, abb^{\dag}, a(e-bb^\dag)\in \mathcal{R}^{\dag}$.
Let $c\in \mathcal{R}$ such that $c$ commutes with  $a$ and  $a^*$, $cab=ab$ and $c^*ab=ab$.
Then, the following statements are equivalent:

\item $(i)$  $bb^\dag a^*ab=a^*ab$;

\item $(ii)$ $b\{1,3\}\cdot a\{1,3\}\cdot c\subseteq (ab)\{1,3\}$;

\item $(iii)$ $b^\dag a^\dag c\in (ab)\{1,3\}$;

\item $(iv)$ $b^\dag a^\dag  c\in (ab)\{1,2,3\}$.

\end{thm}

\begin{proof}Under the conditions of the theorem, using the matrix representations given in Remark \ref{rema8},
 it is not difficult to prove that $b^\dag a^\dag c=b^\dag a^\dag$ and that
necessary and sufficient condition for  (ii) to holds is the fact that
$b\{1,3\}\cdot a\{1,3\}\subseteq (ab)\{1,3\}$. In particular, it is enough to prove the
equivalences among statements (i)-(iv) for the case $c=e$.
Now, the proof of this case follows by  Theorem 3.1 in Cvetkovi\' c-Ili\' c, Harte (2011), where the same conditions
of statements (i)-(iv) were considered for $a$, $b$ two $C^*$-algebra elements
and $c=e$. However, for the sake of completeness the proof of the case $c=e$ will be presented.\par 

We will show that $(i)\Rightarrow(ii)\Rightarrow(iii)\Rightarrow(i)$ and then
$(i)\Rightarrow(iv)\Rightarrow(iii)$. Note that the notation of Remark \ref{rema8} will be used.
In particular, under the hypothesis of the theorem, $a_1$, $a_2\in \mathcal{R}^{\dag}$.\par

$(i)\Rightarrow(ii)$. Suppose that $bb^\dag a^*ab=a^*ab$, which is equivalent to $a_2^*a_1=0$, i.e., $a_1^*a_2=0$. For arbitrary $a^{(1,3)}, b^{(1,3)}$ we have that
$$
abb^{(1,3)}a^{(1,3)}ab=\bmatrix{cc}a_1z_1a_1b&0\\0&0\endbmatrix_{r,q}.
$$
Let $s=a_1a_1^\dag$. Since $a_1^*a_2=0$, $d\in s\A s+(e-s)\A (e-s)$, and then we have that $d^\dag \in s\A s+(e-s)\A (e-s)$. Now,
$a_1^*d^\dag a_2\in \A s\cdot( s\A s+(e-s)\A (e-s))\cdot (e-s)\A=\{0\}$. Hence, $a_1^*d^\dag a_2=0$, i.e., $a_2^*d^\dag a_1=0$.

Since,

\begin{eqnarray*}
&&a_1z_1a_1=a_1a_1^*d^\dag a_1+a_1(e-a_1^*d^\dag a_1)x_1a_1\\
&&\phantom{a_1z_1a_1}=(d-a_2a_2^*)d^\dag a_1+(a_1-(d-a_2a_2^*)d^\dag a_1)x_1a_1\\
&&\phantom{a_1z_1a_1}=a_1,
\end{eqnarray*}

it follows that $abb^{(1,3)}a^{(1,3)}ab=ab$. To prove that $abb^{(1,3)}a^{(1,3)}$ is hermitian, it is sufficient to prove that $a_1z_1$ is hermitian and $a_1z_2=0$. 
By computation, we get that $a_1z_1=a_1a_1^*d^\dag=a_1a_1^*(a_1a_1^*)^\dag$ which is hermitian. Also,

\begin{eqnarray*}
&& a_1z_2=(a_1-a_1a_1^*d^\dag a_1)x_2-a_1a_1^*d^\dag a_2x_4\\
&&\phantom{a_1z_2}=(a_1-(d-a_2a_2^*)d^\dag a_1)x_2=\\
&&\phantom{a_1z_2}=a_2a_2^*d^\dag a_1x_2\\
&&\phantom{a_1z_2}=0.
\end{eqnarray*}

$(ii)\Rightarrow(iii)$. This is evident.

$(iii)\Rightarrow(i)$. From $abb^\dag a^\dag ab=ab$ it follows that $a_1a_1^*d^\dag a_1=a_1$, i.e., $a_2a_2^*d^\dag a_1=0$. Similarly, from
$(abb^\dag a^\dag)^*=abb^\dag a^\dag$, we get that $a_1a_1^*d^\dag $ is hermitian. Now, $d^\dag a_1a_1^*a_1=a_1$, i.e., $a_2a_2^*a_1=0$. Multiplying the last equality by $a_2^\dag$ from the left, we get $a_2^*a_1=0$, which is equivalent to statement $(i)$.

$(iv)\Rightarrow(iii)$. This is obvious.

$(i)\Rightarrow(iv)$. We need to prove that $b^\dag a^\dag abb^\dag a^\dag=b^\dag a^\dag$, which is equivalent to $b^\dag a_1^*d^\dag a_1a_1^*d^\dag =b^\dag a_1^*d^\dag $. The last equality follows from the fact that $d^\dag a_1a_1^*=s$. 
\end{proof}

\begin{thm}\label{thm16} Let $\mathcal{R}$ be a ring with involution.
Consider $a,b \in\mathcal{R}^{\dag}$ such that  $ab,  a^\dag ab, (e-a^\dag a)b\in \mathcal{R}^{\dag}$.
Let $c\in \mathcal{R}$ such that $c$ commutes with $b$ and $b^*$,  $abc=ab$ and  $abc^*=ab$.
Then, the following statements are equivalent:

\item $(i)$   $abb^*a^\dag a= abb^*$;

\item $(ii)$  $c\cdot b\{1,4\}\cdot a\{1,4\}\subseteq (ab)\{1,4\}$;

\item $(iii)$ $cb^\dag a^\dag\in (ab)\{1,4\}$;

\item $(iv)$  $cb^\dag a^\dag\in (ab)\{1,2,4\}$.

\end{thm}
\begin{proof} Apply Theorem \ref{thm15} to $b^*$ and  $a^*$ using that
 $x^*\{1,3\}=(x\{1,4\})^*$ and  $x^*\{1,2,3\}=(x\{1,2,4\})^*$, $x\in \mathcal{R}$.
 \end{proof}
\markright{\hskip4truecm WEIGHTED REVERSE ORDER LAWS }

\begin{rema}\label{rem18} \rm Since the results of this section apply to prime rings with involution,
some of the most relevant examples of these objects will be considered.\par
\indent It is not difficult to prove that a commutative prime ring is just an integral domain.
In addition, given $\mathcal{R}$ a ring with unit element, if  $\mathcal{R}^{n\times n}$ is
the ring of all matrices of order $n\in \mathbb N$ with elements in $\mathcal{R}$,
then according to Theorem 8 in McCoy (1949),
necessary and sufficient for  $\mathcal{R}^{n\times n}$ to be prime is that  $\mathcal{R}$ is prime.
Therefore, if $\mathcal{R}$ is an integral domain and $n\in \mathbb N$, $\mathcal{R}^{n\times n}$
with the transpose is a prime ring with involution.\par
\par 
\indent In the context of Banach algebras, as it has been mentioned in section 1, every algebra of operators which contains
the ideal of all finite rank operators is easily seen to be prime, in particular $L(H)$, the algebra
of all bounded and linear maps defined on the Hilbert space $H$, is a prime ring with involution.
Moreover, if $K(H)$ is the closed ideal of all compact operators defined on the Hilbert space $H$,
then according to Proposition 2.4 in Mathieu (1988), $\mathcal{ C}(H)=L(H)/K(H)$, the Calkin algebra 
of $H$, is a prime ring with involution. Naturally, $L(H)$ and $\mathcal{ C}(H)$ are prime $C^*$-algebras.
 Concerning $C^*$-algebras, in general prime
$C^*$-algebras are not commutative. In fact, if $A=C([0,1])$, then it is not difficult to define
two continuous functions $f$, $g\in C([0,1])$, such that $fg=0$.
On the other hand, prime $C^*$-algebras
were characterized in terms of the norm of elementary operators,
the spectrum of elementary operators, and the Taylor joint spectrum
of left and right multiplication operators, see Proposition 2.3 in Matieu (1989) and 
Theorem 2 and Corollary 3 in Curto, Hern\'andez (1997). 
\end{rema}
\markright{\hskip5truecm BOASSO ET AL. }
\noindent {\bf Acknowledgments} \par
\medskip
\indent The authors wish to express their indebtedness
to the referee, for his observations and suggestions
considerably improved the final version of the present article.\par
\medskip
\noindent {\bf References}\par
\medskip

\noindent   Baksalary J. K.,  Baksalary O. M. (2005).  An invariance property related to the reverse order
 law. {\it Linear Algebra Appl.} 410: 64--69.\par

\noindent  Boasso E.  (2006).  On the Moore-Penrose inverse in
$C^*$-algebras. {\it Extracta Math.} 21:
 93-106.\par

\noindent Curto R. E., Hern\'andez G. C. (1997). A joint spectral characterization
of primeness for $C^*$-algebras. {\it Proc. Amer. Math. Soc.} 125: 3299-3301.

\noindent Cvetkovi\' c-Ili\' c D. S.,  Harte R. E. (2011). Reverse order laws in C*-algebras.
{\it Linear Alg. Appl.} 434: 1388-1394. \par

\noindent  Greville T. N. E. (1966).  Note on the generalized inverse of a
matrix product. {\it SIAM Rev.}   8: 518-521. \par

\noindent  Harte R. E.,  Mbekhta M. (1992).  On generalized inverses in
$C^*$-algebras. {\it Studia Math.}  103:  71-77.\par

\noindent  Harte R., Hern\'andez C. (1998). On the Taylor spectrum of
left-right multipliers. {\it Proc. Amer. Math. Soc.} 126: 397-404.\par

\noindent  Koliha J. J., Patr\'{\i}cio P.  (2002).  Elements of rings with equal spectral idempotents.
{\it J. Aust.  Math. Soc.}  72: 137-152.\par

\noindent  Koliha J. J., Djordjevi\'c D. S., Cvetkovi\' c  D. S. (2007).   Moore-Penrose inverse in rings with
involution. {\it Linear Algebra Appl.}  426: 371-381.\par

\noindent Mathieu M. (1988). Elementary operators on prime $C^*$-algebras II.   
{\it Glasg. Math. J.} 30: 275-284.\par

\noindent Mathieu M. (1989). Elementary operators on prime $C^*$-algebras I.   
{\it Math. Ann.} 284: 223-244.\par

\noindent McCoy N. H. (1949). Prime ideals in general rings. {\it Amer. J. Math.} 71: 823-833.\par

\noindent Moore E. H. (1920). On the reciprocal of the general algebraic matrix.
{\it Bull. Amer. Math. Soc.} 26: 394-395. \par

\noindent Mosi\' c D.,  Djordjevi\' c D. S. (2009).  Moore-Penrose-invertible normal and hermitian
elements in rings. {\it Linear Algebra Appl.}  431: 732-745.\par

\noindent Penrose R.  (1955).  A generalized inverse for matrices.
{\it Proc. Cambridge Philos. Soc.}  51:  406-413.\par
\markright{\hskip4truecm WEIGHTED REVERSE ORDER LAWS }

\noindent  Roch S., Silbermann B. (1999).  Continuity of generalized inverses in Banach algebras. {\it Studia  Math.}
 136: 197-227.\par
\bigskip

\noindent \normalsize \rm Enrico Boasso\par
  \noindent  E-mail: enrico\_odisseo@yahoo.it \par
\medskip
\noindent 
\end{document}